\documentclass[11pt]{amsart}
\usepackage{pstricks,pst-plot}

\definecolor{Black}{cmyk}{0,0,0,1}
\definecolor{OrangeRed}{cmyk}{0,0.6,1,0}            
\definecolor{DarkBlue}{cmyk}{1,1,0,0.20}
\definecolor{myblue}{rgb}{0.66,0.78,1.00}
\definecolor{Violet}{cmyk}{0.79,0.88,0,0}
\definecolor{Lavender}{cmyk}{0,0.48,0,0}

\parskip=\smallskipamount

\newtheorem{theorem}{Theorem}[section]
\newtheorem{lemma}[theorem]{Lemma}
\newtheorem{corollary}[theorem]{Corollary}

\theoremstyle{definition}

\newcommand{\bea}{\begin{eqnarray*}}
\newcommand{\eea}{\end{eqnarray*}}

\numberwithin{equation}{section}

\begin{document}

\title[]{A characterization of the ball}
\author[K. Diederich]{K. Diederich}
\address{K. Diederich: Universitet Wuppertal, Mathematik, Gausstrasse 20, 42119 Wuppertal, Germany.} \email{klas@math.uni-wuppertal.de}
\author[J. E. Forn\ae ss]{J. E. Forn\ae ss}
\address{J. E. Forn\ae ss: Department of Mathematical Sciences, Norwegian University of Science and Technology
7491 Trondheim, Norway} \email{john.fornass@math.ntnu.no}
\author[E. F. Wold]{E. F. Wold}
\address{E. F: Wold: Department of Mathematics, University of Oslo, Postboks 1053 Blindern, 0316 Oslo,
 Norway} \email{erlendfw@math.uio.no}
\date{\today}
\footnote{The second author was supported in part by the Norwegian Research Council grant number 240569 and NSF grant DMS1006294}
\footnote{The third author was supported in part by the Norwegian Research Council grant number 240569}
\maketitle

\begin{abstract}
We study bounded domains with certain smoothness conditions and the properties of their squeezing functions
in order to prove that the domains are biholomorphic to the ball.
 \end{abstract}

\section{Introduction}

Let $\Omega$ be a bounded domain in $\mathbb C^n$. For $z\in \Omega$ let $f_z:\Omega\rightarrow
\mathbb B(0,1)$ be any 1-1 holomorphic map to the unit ball which maps $z$ to the origin.
Let $S_{\Omega,f_z}(z)=\sup \{r>0; \mathbb B(0,r)\subset f(\Omega)\}.$ We define the squeezing
function $S=S_{\Omega}:\Omega \rightarrow (0,1]$ by setting
$S(z)=\sup_{f_z} \{S_{\Omega,f_z}\}.$  See \cite{DengGuanZhang12}, \cite{DengGuanZhang16},
\cite{KimZhang15}, \cite{LiuSunYau04}, \cite{LiuSunYau05}, \cite{Yeung09} and references therein for
results on the squeezing function.

In \cite{FW} Forn\ae ss and Wold proved the following estimate for strongly pseudoconvex domains with smooth
boundary.

\begin{theorem}
Let $\Omega$ be a bounded strongly pseudoconvex domain with $\mathcal C^4$ boundary in $\mathbb C^n$.
Then there exists a constant $C>0$ so that the squeezing function $S_{\Omega}(z)$ satisfies the estimate
$S_{\Omega}(z)\geq 1-Cd(z)$ on $\Omega$ where $d(z)$ denotes the boundary distance.
\end{theorem}

Here we show that this estimate is sharp: Recall that the squeezing function of the unit ball is identically equal to $1.$ In fact if the squeezing function has the value one at at least one point, then the domain is known to be biholomorphic to the ball.

\begin{theorem}
Let $\Omega$ be a bounded  domain with $\mathcal C^2$ boundary in $\mathbb C^n$.
Suppose there does not exist a constant $c>0$ so that the squeezing function $S_{\Omega}(z)$ satisfies the estimate
$S_{\Omega}(z)\leq 1-cd(z)$ on $\Omega$. Then $\Omega$ is biholomorphic to the ball.
\end{theorem}

In the second section, we prove Theorem 1.2. In the third section we show that the theorem fails for domains with 
only $\mathcal C^{1}$ boundary.

\section{Proof of the Theorem}

Theorem 1.3 is equivalent to the following result:\\

\begin{theorem}
Let $\Omega$ be a bounded  domain with $\mathcal C^2$  boundary. Suppose there is a sequence of points $p_i$ approaching the boundary so that the squeezing function $S(p_i)\geq 1-\epsilon_id(p_i), \epsilon_i\rightarrow 0$. Then $\Omega$ is biholomorphic to the ball.
\end{theorem}

\begin{proof}

Say $0\in \Omega.$
Let $\Phi_i:\Omega\rightarrow \mathbb B(0,1)$ be 1-1 holomorphic maps  so that $\Phi_i(p_i)=0$ and the image
contains the ball of radius $S(p_i).$

We collect some lemmas. 

\begin{lemma} 
If $\Omega$ is a bounded  domain with $\mathcal C^2$ boundary, then
there is a constant $C$ so that the Kobayashi distance from $0$ to $p_i$ satisfies the
estimate $d_{K,\Omega}(0,p)\leq \frac{1}{2}\log\frac{1}{ d(p)}+C$
\end{lemma}

We prove this by choosing a curve from $0$ to $p_i$ which ends as a straight normal line at $p_i.$ Then we compare the infinitesimal Kobayashi metric on $\Omega$ with the metric on the intersection with the complex normal line.

Recall the Kobayashi distance on the unit ball:

$$
d_{K,\mathbb B(0,1)}(0,z) = \frac{1}{2}\log \frac{1+\|z\|}{1-\|z\|}.
$$



\begin{lemma}
For points $z\in B(0,1-\epsilon_id(p_i)),$ we have that 
$$
\frac{1}{2}\log \frac{1+\|z\|}{1-\|z\|}\leq d_{K,\Phi_i(\Omega)}(0,z)
$$
\end{lemma}

\begin{lemma}
For all $i, \Phi_i(0)\in \mathbb B(0,1-d(p_i)/e^{2C}).$
\end{lemma}

\begin{proof}
Let $\|z\|=1-d(p_i)/e^{2C}.$

Then 
\bea
d_{K,\Phi_i(\Omega)}(0,z) & \geq & \frac{1}{2}\log \frac{1+\|z\|}{1-\|z\|}\\
& = & \frac{1}{2}\log \frac{1+(1-\frac{d(p_i)}{e^{2C}})}{1-(1-\frac{d(p_i)}{e^{2C}})}\\
& = & \frac{1}{2}\log \frac{2-\frac{d(p_i)}{e^{2C}}}{\frac{d(p_i)}{e^{2C}}}\\
& > & \frac{1}{2}\log \frac{e^{2C}}{d(p_i)}\\
& = & \frac{1}{2}\log \frac{1}{d(p_i)}+C\\
& \geq & d_{K,\Omega}(0,p_i)\\
& = & d_{K,\Phi_i(\Omega)}(0,\Phi_i(0))\\
\eea

Hence any path connecting $0$ to $\Phi_i(0)$ which passes through a point on the boundary of the ball
$\mathbb B(0,1-d(p_i)/e^{2C}$ is too long compared to the Kobayashi distance from $0$ to $\Phi_i(0).$

\end{proof}

We can assume that $\Phi_i(0)=(r,0,\dots,0), 0\leq r<1-d(p_i)/C.$ Define
$$
\Psi_i(z_1,\dots,z_n)=\left(\frac{z_1-r}{1-z_1r},\frac{\sqrt{1-r^2}z_2}{1-z_1r},\dots,
\frac{\sqrt{1-r^2}z_n}{1-z_1r}\right).
$$
Then $\Psi_i$ is an automorphism of the unit ball and the map $F_i=\Psi_i\circ \Phi_i$ is a 1-1 holomorphic
map on $\Omega$ into the unit ball which maps $0$ to $0.$ 

\begin{lemma}
$F_i(\Omega)\supset \mathbb B(0,1-6C\epsilon_i).$ 
\end{lemma}

\begin{proof}
We know that $\Phi_i(\Omega)\supset \mathbb B(0,1-\epsilon_id(p_i)).$ To prove the lemma it suffices
to prove that if $\|z\|=1-2\epsilon_i d(p_i),$ then $\|\Psi_i(z)\|\geq 1-6C\epsilon_i.$
Suppose that $\|z\|=1-2\epsilon_id(p_i).$ Then

\bea
\|\Psi_i(z)\|^2 & = & \frac{(z_1-r)(\overline{z}_1-r)+(1-r^2)(|z_2|^2+\cdots +|z_n|^2)}{|1-z_1r|^2}\\
 & = & \frac{(z_1-r)(\overline{z}_1-r)+(1-r^2)(1-2\epsilon_i d(p_i)^2-|z_1|^2)}{|1-z_1r|^2}\\
& = & \frac{(z_1-r)(\overline{z}_1-r)+(1-r^2)(1-|z_1|^2)}{|1-z_1r|^2}\\
& + & \frac{(1-r^2)(-4\epsilon_i d(p_i)-4\epsilon_i^2 d^2(p_i))}{|1-z_1r|^2}\\
& = & 1-\frac{(1-r^2)(4\epsilon_i d(p_i)+4\epsilon_i^2 d^2(p_i))}{|1-z_1r|^2}\\
& \geq & 1-\frac{(1-r^2)(5\epsilon_i d(p_i))}{(1-r)^2}\\
& \geq & 1-\frac{10\epsilon_i d(p_i)}{1-r}\\
& \geq & 1-\frac{10C\epsilon_i d(p_i)}{d(p_i)}\\
& = & 1-10C\epsilon_i\\
&\Rightarrow & \\
\|\Psi_i(z)\| & \geq & 1-6C\epsilon_i\\
\eea

\end{proof}

\begin{corollary}
$S(0)=1$
\end{corollary}

\begin{corollary}
$\Omega$ is biholomorphic to the unit ball.
\end{corollary}

\end{proof}

\section{An example}

Let $\Omega'$ be a $\mathcal C^\infty$ domain in the right half plane where the boundary contains an interval
$(-i,i)$ on the imaginary axis and which is a topological annulus.
We define $\Omega=\Phi(\Omega')$ there $\Phi(z)=z\log z.$ The squeezing function on $\Omega'$ satisfies
the estimate $S(z)\geq 1-Cd(z)$ since $\Omega'$ is strongly pseudoconvex. The  squeezing function is a biholomorphic invariant and the derivative of $\Phi$ goes to zero when we approach the origin. Hence
the squeezing function of $\Omega$ will not satisfy the estimate $S_{\Omega}\leq 1-cd$ for any $c>0.$
However, the domain is a toplogical annulus so cannot be biholomorphic to the ball. This shows that Theorem 1.2
fails if we only assume  that the boundary is $\mathcal C^1.$

\end{document}